\documentclass[10pt]{amsart}

\usepackage{amsmath}
\usepackage{amsfonts}
\usepackage{amsthm}
\usepackage{amssymb}
\usepackage{enumerate}
\usepackage{tikz-cd}
\usepackage{graphicx}
\usepackage{footmisc}
\usepackage{enumerate}
\usepackage{centernot}
\usepackage{mathtools}
\usepackage{stmaryrd}
\usepackage{url}
\usepackage{multicol}
\usepackage{pdfpages}
\usepackage{colortbl}

\usepackage[style=numeric]{biblatex}
\newtheorem{definition}{Definition}[section]
\newtheorem{theorem}[definition]{Theorem}
\newtheorem{lemma}[definition]{Lemma}
\newtheorem{corollary}[definition]{Corollary}
\newtheorem{proposition}[definition]{Proposition}

\makeatletter
\newcommand{\newreptheorem}[2]{\newtheorem*{rep@#1}{\rep@title}\newenvironment{rep#1}[1]{\def\rep@title{#2 \ref*{##1}}\begin{rep@#1}}{\end{rep@#1}}}
\makeatother

\theoremstyle{definition}

\newtheorem{example}[definition]{Example}
\newtheorem{nonexample}[definition]{Nonexample}
\newtheorem{remark}[definition]{Remark}

\newcommand{\gal}{\mathrm{Gal}}
\newcommand{\nodiv}{\!\!\not|\,}

\newcommand{\rank}{\mathrm{rank}}
\newcommand{\ord}{\mathrm{ord}}
\newcommand{\im}{\mathrm{im}}
\newcommand{\cl}{\mathrm{Cl}}

    \DeclareFontFamily{U}{wncy}{}
    \DeclareFontShape{U}{wncy}{m}{n}{<->wncyr10}{}
    \DeclareSymbolFont{mcy}{U}{wncy}{m}{n}
    \DeclareMathSymbol{\Sh}{\mathord}{mcy}{"58}
\title{Two-descent on some genus two curves}

\author[T.~Evink]{Tim Evink}
\address{Institute of Algebra and Number Theory, Ulm University, Helmholtzstr.~18, 89081 Ulm, Germany.}
\email{tim.evink@uni-ulm.de}
\author[G.J.~van der Heiden]{Gert-Jan van der Heiden}
\address{Center for Contemporary European Philosophy,
P.O.~Box~9103, 6500 HD Nijmegen, The Netherlands.}
\email{g.vanderheiden@ftr.ru.nl}
\author[J.~Top]{Jaap Top}
\address{Bernoulli Institute, University of Groningen,
Nijenborgh~9, 9747 AG Groningen, the Netherlands.}
\email{j.top@rug.nl}
 
\begin{document}

\maketitle

\section{Introduction}

In this paper we study for an arbitrary prime number $p$ the curve $C_p$ of genus $2$ defined by the equation
\begin{equation}\label{eq:thecurve}
y^2=x(x^2-p^2)(x^2-4p^2).
\end{equation}
Specifically, we start by bounding the rank of its Jacobian $J_p$ over $\mathbb{Q}$ in terms of the $2$-Selmer group $S^2(J_p/\mathbb{Q})$. Next we show for three infinite 
sets of prime numbers $p$ how to improve the upper bound on $\rank\,J_p(\mathbb{Q})$ by using a $2$-Selmer group computation over 
$\mathbb{Q}(\sqrt{\pm p})$ of the Jacobian of the curve 
$C=C_1$ defined by 
$y^2=x(x^2-1)(x^2-4)$.
This computation applies the R\'edei symbols of \cite{stevenhagen2018redei}.  
The improved
upper bound yields cases where the
Shafarevich-Tate group $\Sh(J_p/\mathbb{Q})$ is nontrivial.
As an example: for primes $p\equiv 23\bmod 48$ it turns out
that $J_p(\mathbb{Q})$ is finite and $\Sh(J_p/\mathbb{Q})[2]\cong(\mathbb{Z}/2\mathbb{Z})^2$.

We also discuss the $\mathbb{Q}$-rational points of the curve $C_p$. This is easy in
case the group $J_p(\mathbb{Q})$ is finite (as occurs,
for example, for all primes $p\equiv 7\bmod 24$). A less
obvious case we treat is $p=241$; the group
$J_{241}(\mathbb{Q})$ turns out
to have rank $2$. Using so-called `Two-Selmer sets', it is shown that
$C_{241}(\mathbb{Q})$ consists of only the
obvious Weierstrass points (the one at infinity and the ones with $y=0$).

Studying genus $1$ curves depending on a prime number $p$ is
a very classical subject; the survey paper \cite{Na29} 
already lists
various examples; more recent ones are found, e.g.,  
in \cite{BrCa1984}, \cite{Mo1992}, \cite{StrTo1994}. 
The natural question of investigating analogous ideas 
in the case of genus $2$
curves so far seems to have obtained less attention.
The 1998 master's thesis \cite{heiden} by one of us provides a first step
(not yet involving Shafarevich-Tate groups). As shown in
 {\sl loc.\ sit.} Prop.~4.3.3 and Thm.~4.3.4, this already
suffices to conclude for the curves discussed in the
present paper that $C_p(\mathbb{Q})$ consists of the $6$
Weierstrass points only, whenever $p\equiv 7\bmod 24$.
The recent preprint \cite{hima2019} studies some similar
families of genus $2$ curves, but with only $2$ rational
Weierstrass points. Again, computing the $2$-Selmer group
over $\mathbb{Q}$ allow the authors to identify congruence conditions on
the prime $p$ such that the corresponding Mordell-Weil group
is finite. As a consequence, for those primes the only rational points on
the curve are the rational Weierstrass points.

Many results in the present paper originate from two
master's projects \cite{heiden}, \cite{evink} (1998
resp. 2019) by the second and the first author,
supervised by the third one.

\section{Notation and results}\label{NotandRes}
The first step in order to obtain information on the rank of Jacobian $J_p$ of the hyperelliptic curve $C_p$ defined by the equation
\[
y^2=x(x^2-p^2)(x^2-4p^2)
\]
for a prime $p$, is the relatively basic computation of the $2$-Selmer group of $J_p/\mathbb{Q}$. It fits in the well know short exact sequence
\begin{equation}\label{eq:2selmerexact}
0\to J_p(\mathbb{Q})/2J_p(\mathbb{Q})\to S^2(J_p/\mathbb{Q})\to\Sh(J_p/\mathbb{Q})[2]\to 0.
\end{equation}
This Selmer group was computed in \cite{heiden} (with
minor corrections in \cite[Appendix~B]{evink}). 
The computation is based on the method described in
\cite{schaefer1995} and uses (see \cite[Section~7]{FPS1997})
\[
\#J_p(K_v)/2J_p(K_v)=|2|_v^{-2}\cdot \#J_p(K_v)[2]
=|2|_v^{-2}\cdot 16
\]
where $K_v\supset\mathbb{Q}_\ell$ is a
finite extension with
valuation ring $O_v$, and \[|2|_v=\text{vol}(2O_v)/\text{vol}(O_v)=\left\{
\begin{array}{ll}1 & \text{if}\;\ell\neq 2,\\
2^{-[K_v:\mathbb{Q}_2]} & \text{if}\;\ell=2\end{array}\right..
\]The result is as follows. (A calculation illustrating this type of result is the proof of
Lemma~\ref{lem5.1} below.)
\begin{proposition}\label{prop2.1}
For a prime number $p>3$, the $\mathbb{F}_2$-vectorspace
$S^2(J_p/\mathbb{Q})$ of the Jacobian $J_p$ of the curve
defined by $y^2=x(x^2-p^2)(x^2-4p^2)$ has dimension as given
in the next table.
\[
\begin{array}{c|c}
   p \bmod 24  & \dim_{\mathbb{F}_2}S^2(J_p/\mathbb{Q}) \\
   \hline 
    1 & 8 \\
    5,11,13,19 & 5 \\
    7 & 4 \\
    17,23 & 6
\end{array}
\]
\end{proposition}
Since all Weierstrass points on $C_p$ are
$\mathbb{Q}$-rational, one has
$J_p(\mathbb{Q})[2]\cong(\mathbb{Z}/2\mathbb{Z})^4$. Either
by observing that in the present situation $J_p(\mathbb{Q})[2]\hookrightarrow S^2(J_p/\mathbb{Q})$, 
or using that the torsion subgroup $J_p(\mathbb{Q})_{\mbox{\scriptsize tor}}\subset J_p(\mathbb{Q})$ yields a $4$-dimensional subspace
$J_p(\mathbb{Q})_{\mbox{\scriptsize tor}}/2J_p(\mathbb{Q})_{\mbox{\scriptsize tor}}$ 
of $J_p(\mathbb{Q})/2J_p(\mathbb{Q})$, 
the short exact sequence \eqref{eq:2selmerexact} implies
\begin{equation}\label{eq:rankshasel}
\rank\,J_p(\mathbb{Q})+\dim_{\mathbb{F}_2}\Sh(J_p/\mathbb{Q})[2]=\dim_{\mathbb{F}_2}S^2(J_p/\mathbb{Q})-4.
\end{equation}
We state an immediate consequence of this:
\begin{corollary}\label{cor2.2}
For any prime number $p\equiv 7\bmod 24$ one has 
$\rank\,J_p(\mathbb{Q})=0$ and $C_p(\mathbb{Q})$ consists of
only the $6$ Weierstrass points of $C_p$.
\end{corollary}
\begin{proof}
The proof of the statement about the rank is indicated above.
Note that for $p\neq 5$ one has $\#J_p(\mathbb{F}_5)=16$ independent of $p$. Moreover $\#J_5(\mathbb{F}_7)=48$ and
$\#J_5(\mathbb{F}_{11})=128$. Since for primes $\ell\geq 3$ the reduction mod $\ell$
map is an injective group homomorphism on rational torsion points,
it follows that $J_p(\mathbb{Q})$ has torsion subgroup
$(\mathbb{Z}/2\mathbb{Z})^4$ for every prime $p$.
Embedding $C_p\subset J_p$ via $P\mapsto [P-\infty]$ with
$\infty\in C_p$ the Weierstrass point at infinity, one
concludes that $C_p\cap J_p(\mathbb{Q})$ consists of
the divisor classes $[W-\infty]$ for $W$ any Weierstrass point
on $C_p$, implying the result. 
\end{proof}

For the primes $p\equiv 5, 11, 13, 19\bmod 24$ 
the structure of the group $J_p(\mathbb{Q})$
is in fact also predicted by Proposition~\ref{prop2.1}:

\begin{corollary}
For any prime $p>3$, assume that $\Sh(J_p/\mathbb{Q})$ is finite. Then
\[\text{rank}\,J_p(\mathbb{Q})\equiv
\left\{
\begin{array}{lcl}
 1\bmod 2 &&\text{if}\;\;p\equiv 5, 11, 13, 19\bmod 24;\\
 0\bmod 2 && otherwise.
\end{array}
\right.
\]
In particular, if for a prime $p\equiv 5, 11, 13, 19\bmod 24$
the group $\Sh(J_p/\mathbb{Q})$ is finite then for this prime
$J_p(\mathbb{Q})\cong \mathbb{Z}\times (\mathbb{Z}/2\mathbb{Z})^4$.
\end{corollary}
\begin{proof}
By a result of Poonen and Stoll \cite[\S6, \S8]{poonenstoll1999}
finiteness of $\Sh$ and the fact that $C_p$
contains a rational point, implies that 
$\dim_{\mathbb{F}_2}\Sh(J_p/\mathbb{Q})[2]$ is even. 
Hence Equation \eqref{eq:rankshasel}
and Proposition~\ref{prop2.1} imply the first assertion
as well as $\rank\,J_p(\mathbb{Q})=1$
whenever $p\equiv 5, 11, 13, 19\bmod 24$.
The result follows since the proof of
Corollary~\ref{cor2.2} in particular determines the torsion
subgroup of $J_p(\mathbb{Q})$.
\end{proof}
The remainder of this paper deals with improvements of
Proposition~\ref{prop2.1} and variations on Corollary~\ref{cor2.2}. Specifically, this is possible in all remaining
congruence classes (so, $p\equiv 1, 17, 23\bmod 24$). We show the following.
\begin{theorem}\label{Thm23mod24}
Let $p\equiv 23\bmod 48$ be a prime number. The Jacobian
$J_p$ of the curve corresponding to $y^2=x(x^2-p^2)(x^2-4p^2)$
satisfies $J_p(\mathbb{Q})=J_p(\mathbb{Q})[2]\cong
(\mathbb{Z}/2\mathbb{Z})^4$ and $\Sh(J_p/\mathbb{Q})[2]\cong
(\mathbb{Z}/2\mathbb{Z})^2$.
\end{theorem}

\begin{theorem}\label{Thm17mod24}
Let $p\equiv 17\bmod 24$ be a prime number that does not
split completely in $\mathbb{Q}(\sqrt[4]{2})$. The Jacobian
$J_p$ of the curve corresponding to $y^2=x(x^2-p^2)(x^2-4p^2)$
satisfies $J_p(\mathbb{Q})=J_p(\mathbb{Q})[2]\cong
(\mathbb{Z}/2\mathbb{Z})^4$ and $\Sh(J_p/\mathbb{Q})[2]\cong
(\mathbb{Z}/2\mathbb{Z})^2$.
\end{theorem}

\begin{theorem}\label{Thm1mod24}
Let $p\equiv 1\bmod 24$ be a prime number satisfying one
of the conditions
\begin{itemize}
    \item[$(a)$] $p$ splits completely in $\mathbb{Q}(\sqrt[4]{2})$
    and not in $\mathbb{Q}(\sqrt{1+\sqrt{3}})$;
    \item[$(b)$] $p\equiv 1\bmod 48$ and $p$ splits completely in
    $\mathbb{Q}(\sqrt{1+\sqrt{3}})$ and not in
    $\mathbb{Q}(\sqrt[4]{2})$;
    \item[$(c)$] $p\equiv 25\bmod 48$ and $p$ does not split
    completely in either of $\mathbb{Q}(\sqrt[4]{2})$
    and  $\mathbb{Q}(\sqrt{1+\sqrt{3}})$.
\end{itemize}
The Jacobian
$J_p$ of the curve corresponding to $y^2=x(x^2-p^2)(x^2-4p^2)$
satisfies $J_p(\mathbb{Q})=J(\mathbb{Q})[2]\cong
(\mathbb{Z}/2\mathbb{Z})^4$ and $\Sh(J_p/\mathbb{Q})[2]\cong
(\mathbb{Z}/2\mathbb{Z})^4$.
\end{theorem}
Using Chebotar\"{e}v's density theorem (see, e.g., \cite{LS}),
one observes that the set of prime numbers satisfying
the condition given in Theorem~\ref{Thm23mod24} has
a positive Dirichlet density. The same holds for the set of
primes satisfying the condition in Theorem~\ref{Thm17mod24}
and for each of the three sets corresponding to
Theorem~\ref{Thm1mod24}~(a), \ref{Thm1mod24}~(b), 
and \ref{Thm1mod24}~(c).
\section{R\'edei Symbols}

In this section we recall the definition
and various properties of the R\'edei symbol. It is a tri-linear symbol taking values in $\mu_2$ and it satisfies a reciprocity law based on the product formula for quadratic Hilbert symbols. This reciprocity allows us to link the splitting behaviour of certain primes in dihedral extensions over $\mathbb{Q}$ of degree $8$ in a non-trivial way, which functions as a useful supplement to various $2$-Selmer group computations. The reciprocity of the R\'edei-symbol is a recent result due to P. Stevenhagen in \cite{stevenhagen2018redei}; his text is the basis for the exposition in this section.\\

Let $a,b$ be square-free integers representing non-trivial elements in $\mathbb{Q}^{*}/\mathbb{Q}^{*2}$, 
and suppose their local quadratic Hilbert symbols are all trivial:
\begin{equation}\label{eq:quadhilbert}
(a,b)_p=1,\quad\text{for all primes}\; p.
\end{equation}
By the local-global principle of Hasse and Minkowski, condition \eqref{eq:quadhilbert} is equivalent to the existence of a non-zero rational solution $(x,y,z)$ to the equation
\begin{equation}\label{eq:hassmink}
x^2-ay^2-bz^2=0.
\end{equation}
Take such a solution and put 
\begin{equation}
\alpha=2(x+z\sqrt{b}),\quad \beta=x+y\sqrt{a}.
\end{equation}
Then $F:=E(\sqrt{\alpha})=E(\sqrt{\beta})$ defines a
quadratic extension of $E=\mathbb{Q}(\sqrt{a},\sqrt{b})$ that is normal over $\mathbb{Q}$, cyclic of degree $4$ over $\mathbb{Q}(\sqrt{ab})$, and dihedral of degree $8$ over $\mathbb{Q}$ when $\mathbb{Q}(\sqrt{ab})\neq\mathbb{Q}$, see \cite[Lemma~5.1, Corollary~5.2]{stevenhagen2018redei}. The extension $F$ can be twisted to $F_t$ for $t\in\mathbb{Q}^{*}$ by scaling the solution $(x,y,z)$ to $(tx,ty,tz)$. 
By \cite[Propositions~7.2,7.3]{stevenhagen2018redei} 
choosing $t$ appropriately ensures that $F_t/E$ is unramified at all finite primes of odd residue
characteristic, but in some cases ramification over $2$ cannot be avoided. With  $\Delta(d)=\Delta(\mathcal{O}_{\mathbb{Q}(\sqrt{d})})$ for $d\in\mathbb{Q}^{*}/\mathbb{Q}^{*2}$ denoting the discriminant, one makes the following definition.
\begin{definition}\label{minram}
Let $K=\mathbb{Q}(\sqrt{ab})$ for non-trivial $a,b\in\mathbb{Q}^{*}/\mathbb{Q}^{*2}$, and let $F$ be the quadratic extension of $E=\mathbb{Q}(\sqrt{a},\sqrt{b})$ corresponding with a solution of \eqref{eq:hassmink}. The extension $F/K$ is minimally ramified if the following conditions hold:
\begin{itemize}
    \item[$(a)$] The extension $F/K$ is unramified over all odd primes $p\nodiv\mathrm{gcd}(\Delta(a),\Delta(b))$.
    \item[$(b)$] The extension $F/K$ is unramified over $2$ if $\Delta(a)\Delta(b)$ is odd, or if one of $\Delta(a),\Delta(b)$ is $1\bmod 8$.
    \item[$(c)$] If $\{\Delta(a),\Delta(b)\}\equiv \{4,5\}\bmod 8$, take $s\in\{a,b\}$ such that $\Delta(s)\equiv 4\bmod 8$. The local biquadratic extension $\mathbb{Q}_2(\sqrt{s})\subset F\otimes\mathbb{Q}_2$ must have conductor $2$. 
\end{itemize}
\end{definition}
By \cite[Lemma~7.7]{stevenhagen2018redei}  it is possible to twist a given $F$ to a suitable $F_t$ which is minimally ramified over $\mathbb{Q}(\sqrt{ab})$. 
For convenience, a degree $8$ dihedral extension of $\mathbb{Q}$ is called minimally ramified if it is so over its subfield defined by the order $4$ cyclic subgroup of the Galois group. Observe that the definition imposes no restrictions over the prime $2$ in case $2$ is totally ramified in $\mathbb{Q}(\sqrt{a},\sqrt{b})$. 

\begin{definition}
For non-trivial $a,b,c\in\mathbb{Q}^{*}/\mathbb{Q}^{*2}$ with local quadratic Hilbert symbols
\begin{equation}\label{eq:quadhilbs}
(a,b)_p=(a,c)_p=(b,c)_p=1
\end{equation}
for all primes $p$ and moreover
\begin{equation}\label{eq:coprimedisc}
\mathrm{gcd}(\Delta(a),\Delta(b),\Delta(c))=1,
\end{equation}
set $K=\mathbb{Q}(\sqrt{ab})$ and $E=\mathbb{Q}(\sqrt{a},\sqrt{b})$, and take a
corresponding $F/K$ which is minimally ramified. 
Define $[a,b,c]\in\gal(F/E)=\mu_2$ by
\[
[a,b,c]=\begin{cases}
\mathrm{Art}(\mathfrak{c},F/K) & \text{ if } c> 0 \\
\mathrm{Art}(\mathfrak{c}\infty,F/K) & \text{ if } c<0 \\
\end{cases}
\]
where $\mathrm{Art}(\cdot , \cdot)$ is the Artin symbol and $\mathfrak{c}\in\mathcal{I}(\mathcal{O}_K)$ has norm $|c_0|$ with $c_0$ the square-free integer representing
$c$, and $\infty$ denotes an infinite prime of $K$.\\
If at least one of $a,b$ and $c$ is trivial then one sets $[a,b,c]=1$.
\end{definition}
\begin{proposition}
For $a,b,c\in\mathbb{Q}^{*}/\mathbb{Q}^{*2}$ satisfying \eqref{eq:quadhilbs} and \eqref{eq:coprimedisc}, the Rédei symbol $[a,b,c]\in\mu_2$ is well-defined. Moreover, the symbol is tri-linear, and perfectly symmetrical in all three arguments.
\end{proposition}
\begin{proof} {\it (Sketch.)}
If $p|c$, then $(c,b)_p=(c,a)_p=1$ implies that $p$ is either split or ramified in both $\mathbb{Q}(\sqrt{a})$ and $\mathbb{Q}(\sqrt{b})$. 
Condition \eqref{eq:coprimedisc} implies that $p$ cannot ramify in both, hence a prime $\mathfrak{p}_K|p$ in $K$ has norm $p$ and splits in $E$. 
The prime $\mathfrak{p}_K$ is unramified in $F$ by the minimal ramification of $F$, where the parity of $p$ determines whether this is due to condition $(a)$ or $(b)$. It follows that indeed $\mathrm{Art}(\mathfrak{p}_K,F/K)\in\gal(F/E)$, and as $\gal(F/E)$ is in the center of $\gal(F/\mathbb{Q})$ this Artin symbol is independent of $\mathfrak{p}_K$. When $c<0$ and $K$ is real, the Artin symbol in $F$ of any infinite prime of $K$ measures whether $F$ is real or complex and hence is independent of the choice of infinite prime of $K$. As $[a,b,c]$ is the product of such Artin symbols we see that $[a,b,c]$ does not depend on the choice of $\mathfrak{c}$ or $\infty$. For the independence of $F$ we refer to \cite[Corollary~8.2]{stevenhagen2018redei}.\\

The set of triples $(a,b,c)$ in $\mathbb{Q}^{*}/\mathbb{Q}^{*2}$ for which \eqref{eq:quadhilbs} and \eqref{eq:coprimedisc} hold is `tri-linearly closed', and the R\'edei symbol $[a,b,c]$ is clearly linear in $c$, hence tri-linearity follows from the symmetry. The symmetry in the first two arguments is immediate, while the identity
\[
[a,b,c]=[a,c,b]
\]
is a non-trivial reciprocity depending on the product formula for quadratic Hilbert symbols in $\mathbb{Q}(\sqrt{a})$. The proof of this reciprocity is the subject of \cite[Section~8]{stevenhagen2018redei}.
\end{proof}
\begin{example}\label{redeiexample1}
Consider the case when $a=b=2$. 
Then the invariant fields $F$ and $F'$ of the subgroups generated by $-1\bmod 16$ and $7\bmod 16$ inside $\gal(\mathbb{Q}(\zeta_{16})/\mathbb{Q})$, respectively, 
are two minimally ramified extensions of $\mathbb{Q}$ which can be used to compute a R\'edei symbol of the form $[2,2,c]$ provided that the symbol is defined, i.e. when $\Delta(c)$ is odd and $(2,c)_2=1$, i.e. when $c\equiv 1\bmod 8$. Taking for example $c=-p$ for a prime $p\equiv -1\bmod 8$, then as $F$ is totally real and $p$ splits completely in $F$ precisely when $p\equiv \pm 1\bmod 16$ we obtain
\[
[2,2,-p]=\begin{cases}
1 & \text{ if } p\equiv -1\bmod 16 \\
-1 & \text{ if } p\equiv 7\bmod 16 \\
\end{cases}
\]
Note that we get the same conclusion when using the (complex!) field $F'$ as $p$ splits completely in $F'$ precisely when $p\equiv 1,7\bmod 16$.
\end{example}

\begin{nonexample}
Continuing the setup of Example \ref{redeiexample1}, we see that $`[2,2,-1]'$ is not defined as $\Delta(-1)$ is even (although $(2,-1)_2=1$). We can nonetheless consider an Artin symbol `that should define $[2,2,-1]$', but as $F$ is real and $F'$ complex, such a symbol is \textit{not} independent of the minimally ramified extension.
\end{nonexample}

\begin{example}\label{Ex3.6}
Let $p\equiv 1\bmod 8$ be a prime. Let $\pi\in\mathbb{Z}[\sqrt{2}]$ be an element of norm $p$ with conjugate $\pi'$. Then $[2,p,p]=1$ precisely when $\pi$ is a square mod $\pi'$. Since $[2,p,p]=[p,p,2]$, this is equivalent to $2$ being completely split in the quartic subfield $E$ of $\mathbb{Q}(\zeta_p)$. As $E$ corresponds with the subgroup of fourth powers in $\gal(\mathbb{Q}(\zeta_p)/\mathbb{Q})=(\mathbb{Z}/p\mathbb{Z})^{*}$ and $2\bmod p=\mathrm{Frob}_p\in \gal(\mathbb{Q}(\zeta_p)/\mathbb{Q})$, we see that $2$ splits completely in $E$ precisely when $2\bmod p$ is a fourth power, i.e. when $p$ splits completely in $\mathbb{Q}(\sqrt[4]{2})$, i.e. when $[2,-2,p]=1$. We thus have the identity
\[
[2,p,p]=[2,-2,p].
\]
\end{example}
With this we obtain a generalisation of \cite[Prop. ~4.1]{StrTo1994}, where it is used to prove that $ (\mathbb{Z}/2\mathbb{Z})^2\subset\Sh(E/\mathbb{Q})[2]$ for the elliptic curve $E$ defined by $y^2=(x+p)(x^2+p^2)$ for a prime $p\equiv 9\bmod 16$ such that $1+\sqrt{-1}\in \mathbb{F}_p$ is a square.
\begin{corollary}
Let $p\equiv 1\bmod 8$ be a prime, let $\pi\in\mathbb{Z}[\sqrt{2}]$ have norm $p$ with conjugate $\pi'$ and let $i\in\mathbb{F}_p$ be a primitive fourth root of unity. Consider the following statements.
\begin{enumerate}[(a)]
    \item $\pi$ is a square mod $\pi'$.
    \item $1+i$ is a square mod $p$.
\end{enumerate}
Then the statements are equivalent when $p\equiv 1\bmod 16$, while for $p\equiv 9\bmod 16$ exactly one of the statements holds.
\end{corollary}
\begin{proof}
Statement $(a)$ holds when $[2,p,p]=[2,-2,p]=1$, while statement $(b)$ holds when $[2,-1,p]=1$. The result follows because
\[
[2,-2,p]\cdot[2,-1,p]=[2,2,p]=
\begin{cases}
1 & \text{ if }p\equiv 1\bmod 16,\\
-1 & \text{ if }p\equiv 9\bmod 16.
\end{cases}\qedhere
\]
\end{proof}

\section{Computation of $2$-Selmer groups}
We start by recalling the explicit form of $2$-descent that will be used. Let $K$ be a number field and $C$ the hyperelliptic curve defined by $y^2=f(x)$, for $f\in K[x]$ square-free and of odd degree $2g+1$. We have the short exact sequence
\[
0\to J(K)/2J(K)\to S^2(J/K)\to \Sh(J/K)[2]\to 0,
\]
where $S^2(J/K)$ and $\Sh(J/K)$ are respectively the $2$-Selmer group and the Shafarevich-Tate group defined in terms of Galois cohomology by
\begin{align*}
S^2(J/K)&:=\ker\left(H^1(G_K,J(\overline{K})[2])\to \prod\nolimits_{\mathfrak{p}}H^1(G_{K_{\mathfrak{p}}},J(\overline{K_{\mathfrak{p}}}))\right),\\
\Sh(J/K)&:=\ker\left(H^1(G_K,J(\overline{K}))\to \prod\nolimits_{\mathfrak{p}}H^1(G_{K_{\mathfrak{p}}},J(\overline{K_{\mathfrak{p}}}))\right).
\end{align*}
By \cite[Theorems~2.1 \& 2.2]{schaefer1995} one has
$H^1(G_K,J(\overline{K})[2])\cong \ker(A^{*}/A^{*2}\xrightarrow{N} K^{*}/K^{*2})$, 
where $A = K[x]/(f(x))$ and $N$ is induced by the norm map $A\to K$.
This identifies $S^2(J/K)$ with the elements in $ \ker(A^{*}/A^{*2}\xrightarrow{N} K^{*}/K^{*2})$ that are mapped, according to the commutative diagram
\[
\begin{tikzcd}
J(K)/2J(K) \arrow[d] \arrow[r, "\delta", hook]                                    & A^{*}/A^{*2} \arrow[d]                     \\
J(K_{\mathfrak{p}})/2J(K_{\mathfrak{p}}) \arrow[r, "\delta_{\mathfrak{p}}", hook] & A_{\mathfrak{p}}^{*}/A_{\mathfrak{p}}^{*2},
\end{tikzcd}
\]
into $\im(\delta_{\mathfrak{p}})$ for all primes $\mathfrak{p}$ of $K$.

We consider the special case that $f\in\mathcal{O}_K[x]$ is monic and completely splits, so $f=\prod_{i=1}^{2g+1}(x-\alpha_i)$ for distinct $\alpha_j\in \mathcal{O}_K$. In this case $A\xrightarrow{\sim} \bigoplus_{i=1}^{2g+1}K$ determined by $x\mapsto (\alpha_1,\dotsc,\alpha_{2g+1})$, and
the norm map $A\to K$ corresponds to multiplication
$\oplus_{i=1}^{2g+1} K\to K$. Hence the kernel of the norm $\oplus_{i=1}^{2g+1}K^*/{K^*}^2\stackrel{N}{\longrightarrow} K^*/{K^*}^2$ consists of the `hyperplane' of those $(2g+1)$-tuples for which the product of all coordinates is trivial. 

Let $S$ consist of the real primes of $K$ together with the finite primes dividing $2\Delta(f)$, and put $K(S):=\{x\in K^{*}/K^{*2}:\ord_{\mathfrak{p}}(x)\equiv 0\bmod 2\text{ for all finite }\mathfrak{p}\notin S\}$.
One has (compare \cite[pp.~226-227]{schaefer1995})
\begin{equation}\label{eq:finselmalg}
S^2(J/K)\subset\ker\bigg(\bigoplus_{i=1}^{2g+1}K(S)\to K(S)\bigg),
\end{equation}
and $S^2(J/K)$ consists of those elements in the kernel of \eqref{eq:finselmalg} that map into $\im(\delta_{\mathfrak{p}})$ for each $\mathfrak{p}\in S$ in the following diagram.
\vspace{-0.2cm}
\[
\begin{tikzcd}
J(K)/2J(K)\ar{r}{\delta}\ar{d} & \displaystyle\bigoplus\limits_{i=1}^{2g+1}K^{*}/K^{*2}\ar{d}\\
J(K_{\mathfrak{p}})/2J(K_{\mathfrak{p}})\ar{r}{\delta_{\mathfrak{p}}} & \displaystyle\bigoplus_{i=1}^{2g+1}K_{\mathfrak{p}}^{*}/K_{\mathfrak{p}}^{*2}.
\end{tikzcd}
\vspace{-0.1cm}
\]
Here the injective homomorphism $\delta$ and similarly $\delta_{\mathfrak{p}}$
is given by
\begin{equation}\label{eq:imagedivisors}
\sum_{i=1}^r[P_i]-r[\infty]\mapsto\prod_{i=1}^r(x(P_i)-\alpha_1,\dotsc, x(P_i)-\alpha_{2g+1}),
\end{equation}
for $P_1,\dotsc,P_r\in C(\overline{K})$ forming a $G_K$-orbit not containing a Weierstrass point. The $j$-th coordinate of the $\delta$-image of $[(\alpha_i,0)]-[\infty]$ for $i\neq j$ is $\alpha_i-\alpha_j$. The $i$-th coordinate is then determined by the hyperplane condition: it equals $\prod_{j\neq i}(\alpha_i-\alpha_j)$. 
As already remarked in Section~\ref{NotandRes} the cardinality
of $J(K_{\mathfrak{p}})/2J(K_{\mathfrak{p}})$ and hence that of
$\im(\delta_{\mathfrak{p}})$ is known. In practise this makes it
fairly straightforward to describe explicit representants of the
elements in $\im(\delta_{\mathfrak{p}})$, for each $\mathfrak{p}\in S$.


The group $K(S)$ fits in the exact sequence
\[
\begin{tikzcd}
0 \arrow[r] & R_S^{*}/R_S^{*2} \arrow[r] & K(S) \arrow[r,"\beta"] & \cl(R_S)[2] \arrow[r] & 0
\end{tikzcd}
\]
where $R_S=\{0\}\cup\{x\in K^{*}:\mathrm{ord}_{\mathfrak{p}}(x)\geq 0\text{ for all finite }\mathfrak{p}\notin S\}$ is the ring of $S$-integers in $K$. 
Here $\beta$ sends $xK^{*2}$ to the class $[I R_S]$, where $x\mathcal{O}_K=\mathfrak{a}I^2$ with $\mathfrak{a}$ and $I$ co-prime fractional ideals such that $\mathfrak{a}$ is supported on prime ideals of $S$ and the support of $I$ does not contain any prime of $S$.
This is well-known; for completeness see \cite[Prop.~2.4.4]{evink}. The case of interest to us is when $K$ has odd class number. 
\begin{proposition}\label{sunitbasis}
If $K$ has odd class number then the map $R_S^{*}/R_S^{*2}\to K(S)$ is an isomorphism. Moreover, 
for each finite $\mathfrak{p}\in S$ writing $\mathfrak{p}^{k_{\mathfrak{p}}}=(x_{\mathfrak{p}})$ with $k_{\mathfrak{p}}$ the order of $\mathfrak{p}$ in the class group of $K$, the $x_{\mathfrak{p}}$ together with an $\mathbb{F}_2$-basis for $\mathcal{O}_K^{*}/\mathcal{O}_K^{*2}$ form an $\mathbb{F}_2$-basis for $K(S)$.
\end{proposition}
\begin{proof}
A detailed proof of this standard fact is provided in \cite[Cor.~2.4.7]{evink}.
\end{proof}
For an odd prime $p$ write $p^{*}=(-1)^{(p-1)/2}p$, so
$\mathbb{Q}(\sqrt{p^{*}})$ is the quadratic subfield of the cyclotomic field $\mathbb{Q}(\zeta_p)$. 
In what follows we will compute $2$-Selmer groups over these
quadratic fields. One has
\begin{lemma}\label{oddclassnmr}
For any odd prime $p$ the field $K=\mathbb{Q}(\sqrt{p^{*}})$ has odd class number, and if $K$ is real (i.e., $p\equiv 1\bmod 4$) then a fundamental unit of $K$ has norm $-1$.
\end{lemma}
\begin{proof}
For a proof using genus theory, see for example \cite[Thm~2.1]{stevenhagen2018redei}. A slightly more direct argument 
is given in \cite[Appendix~A.2]{evink}.
\end{proof}

\section{Proofs of the rank and Shafarevich-Tate group results}
Consider the genus two hyperelliptic curves 
\[
C/\mathbb{Q}\colon y^2=f(x):=x(x^2-1)(x^2-4),
\]
and, for $p$ any prime number, 
\[ C_p/\mathbb{Q}\colon  y^2=x(x^2-p^2)(x^2-4p^2). \]

Then $C_p$ is a quadratic twist of $C$ over both $\mathbb{Q}(\sqrt{p})$ and $\mathbb{Q}(\sqrt{-p})$. Let $J$ and $J_p$ denote the Jacobians of 
$C$ and $C_p$, respectively. 
Observe that
\begin{equation}\label{eq:rankrelationtwists}
\rank\,J_p(\mathbb{Q})+\rank\,J(\mathbb{Q})=\rank\,J(\mathbb{Q}(\sqrt{\pm p})),
\end{equation}
for both possibilities of the sign $\pm$. A quick computation (Lemma~\ref{lem5.1}) yields $\rank\,J(\mathbb{Q})=0$. Since for the Jacobians at hand
the torsion subgroup yields a subgroup of
the $2$-Selmer group of dimension $4$, 
it follows that
\[\rank\,J_p(\mathbb{Q})\leq
\dim_{\mathbb{F}_2}S^2(J/\mathbb{Q}(\sqrt{\pm p}))-4.\]
Using $\mathbb{Q}(\sqrt{p})$ in case $p\equiv 1,17\bmod 24$ and $\mathbb{Q}(\sqrt{-p})$ for $p\equiv 23\bmod 24$, 
it will be shown that for certain subsets of these primes
the bound for $\rank\,J_p(\mathbb{Q})$ obtained in this way sharpens
the one which follows by directly applying
Proposition~\ref{prop2.1}. Specifically, this 
results in proofs for Theorems \ref{Thm23mod24} - \ref{Thm1mod24}.

\vspace{\baselineskip}\label{NotationDxi}
Label the roots of $f$ as $(\alpha_1,\alpha_2,\alpha_3,\alpha_4,\alpha_5)=(-2,-1,0,1,2)$.
For a field $F\supset\mathbb{Q}$ and a point $(\xi,\eta)\in C(F)$ write $D_{\xi}\in J(F)$ for the point corresponding to the divisor $[(\xi,\eta)]-[\infty]$ on $C$. Note that
although $D_\xi$ depends on $\eta$, its image in
the $2$-Selmer group $S^2(J/F)$ does not.
The image of $J(\mathbb{Q})[2]$ under $\delta$ is spanned by
\[
\begin{array}{c|c c c c c}
& x+2 & x+1 & x & x-1 & x-2 \\
\hline
D_{-2} & 6 & -1 & -2 & -3 & -1 \\
D_{-1} & 1 & -6 & -1 & -2 & -3 \\
D_{0} & 2 & 1 & 1 & -1 & -2 \\
D_{1} & 3 & 2 & 1 & -6 & -1 
\end{array}
\]
Here $x-\alpha_i$ denotes the map 
$[P]-[\infty]\mapsto x(P)-\alpha_i$ as in \eqref{eq:imagedivisors},
compare \cite{schaefer1995}.\\

The local fields for which we need the images 
$\text{im}\,\delta_p$ are $\mathbb{Q}_2$, 
$\mathbb{Q}_3$, $\mathbb{Q}_3(i)$ and $\mathbb{R}$. Much of this was already done in 
\cite[pp.~43-45]{heiden}. One has $\mathbb{Q}_2^{*}/\mathbb{Q}_2^{*2}=\langle -1,2,3\rangle$, $\mathbb{Q}_3^{*}/\mathbb{Q}_3^{*2}=\langle -1,3\rangle$, for $F=\mathbb{Q}_3(i)$ moreover $F^{*}/F^{*2}=\langle 3,r\rangle$, where $r=1+i$, and of course $\mathbb{R}^{*}/\mathbb{R}^{*2}=\langle -1\rangle$. The local images are then spanned as follows.
\begin{equation*}
\begin{array}{c|c c c c c }
\mathbb{Q}_2 & x+2 & x+1 & x & x-1 & x-2 \\
\hline
D_{-2} & 6 & -1 & -2 & -3 & -1 \\
D_{-1} & 1 & -6 & -1 & -2 & -3 \\
D_0 & 2 & 1 & 1 & -1 & -2 \\
D_1 & 3 & 2 & 1 & -6 & -1 \\
D_6 & 2 & -1 & 6 & -3 & 1 \\
D_7 & 1 & 2 & -1 & 6 & -3 
\end{array}
\end{equation*}
\begin{equation*}
\begin{array}{c|c c c c c}
\mathbb{Q}_3 & x+2 & x+1 & x & x-1 & x-2 \\
\hline
D_{-2} & -3 & -1 & 1 & -3 & -1 \\
D_{-1} & 1 & 3 & -1 & 1 & -3 \\
D_0 & -1 & 1 & 1 & -1 & 1 \\
D_4 & 1 & -3 & -1 & 1 & 3 \\
\end{array}
\end{equation*}
\begin{equation*}
\begin{array}{c|c c c c c}
\mathbb{Q}_3(i) & x+2 & x+1 & x & x-1 & x-2 \\
\hline
D_{-2} & 3 & 1 & 1 & 3 & 1 \\
D_{-1} & 1 & 3 & 1 & 1 & 3 \\
D_{i} & r & r & 1 & r & r \\
D_{4+3i} & 3r & 1 & 1 & 3r & 1 \\
\end{array}
\end{equation*}
\begin{equation*}
\begin{array}{c|c c c c c}
\mathbb{R} & x+2 & x+1 & x & x-1 & x-2 \\
\hline
D_{-1} & 1 & -1 & -1 & -1 & -1 \\
D_{0} & 1 & 1 & 1 & -1 & -1 \\
\end{array}
\end{equation*}

\begin{lemma}\label{lem5.1}
We have $\rank\,J(\mathbb{Q})=0$.
\end{lemma}
\begin{proof}
It suffices to show $\dim_{\mathbb{F}_2}S^2(J/\mathbb{Q})=4$. Note $\Delta(f)=2^{10}\cdot 3^{4}$, so $S=\{2,3,\infty\}$ and $K(S)=\langle -1,2,3\rangle$. Then $S^2(J/\mathbb{Q})$ injects into the $2$-adic image, and
\[
S^2(J/\mathbb{Q})=A\oplus\delta(J(\mathbb{Q})[2])
\]
where $A$ consists of all $x\in S^2(J/\mathbb{Q})$ with 
$2$-adic image in the span of
\[
\begin{array}{c c c c c c c}
 ( & 2, & -1, & 6, & -3, &  1   & ),\\
  ( & 1, & 2, & -1, & 6, & -3 & ).
\end{array}
\]
If $x=(e_1,\dotsc,e_5)\in A$, then the $3$-adic image forces $e_3\in\langle -1\rangle$, hence $x$ is in the span of $(1,2,-1,6,-3)$. Therefore $x$ is trivial because
 $(1,2,-1,6,-3)\not\in\im(\delta_3)$. Thus $A=0$ and $S^2(J/\mathbb{Q})$ has $\mathbb{F}_2$-dimension $4$. 
\end{proof}

We now compute $S^2(J/\mathbb{Q}(\sqrt{-p}))$ for $p\equiv 23\bmod 24$ 
and $S^2(J/\mathbb{Q}(\sqrt{p}))$ for $p\equiv 1,17\bmod 24$. The computation follows \cite[\S~3.4.2-4]{evink}, 
except that R\'edei symbols are used instead of
various reciprocity arguments in {\sl loc.\ sit.}\\

Consider a prime $p\equiv 23\bmod 24$ and let $K=\mathbb{Q}(\sqrt{-p})$. Then $K$ is complex and both $2$ and $3$ split in $K$, so as set $S$ of
places of $K$ needed for embedding
$S^2(J/K)$ in $\oplus_{i=1}^5K(S)$ we take
the four primes dividing $6$. The completion of
$K$ at a prime in $S$ equals $\mathbb{Q}_2$ or $\mathbb{Q}_3$.\\
Write $\mathfrak{p}_3,\mathfrak{q}_3$ for the prime ideals in $\mathcal{O}_K$ dividing $3$ and let $k_3$ be the order of $[\mathfrak{p}_3]$ in $\cl_K$. 
Then $\mathfrak{p}_3^{k_3}=(x_3)$ for some $x_3\in \mathcal{O}_K$. 
Since $\mathfrak{q}_3\nmid (x_3)$ and 
$K_{\mathfrak{q}_3}=\mathbb{Q}_3$, this
$x_3$ maps to $\pm 1$ in $K_{\mathfrak{q}_3}^{*}/K_{\mathfrak{q}_3}^{*2}$. Multiplying $x_3$ by $-1$ if necessary, we may and will assume that $x_3$ is a square in
$K_{\mathfrak{q}_3}$. 
The conjugate $y_3\in\mathcal{O}_K$ of $x_3$ satisfies  $\mathfrak{q}_3^{k_3}=(y_3)$ and $x_3y_3=3^{k_3}$.\\
Let $\mathfrak{p},\mathfrak{q}$ be the prime
ideals in $\mathcal{O}_K$ over $2$. In the
$\mathfrak{p}$-adic completion, $x_3$ and $y_3$ 
yield elements of $\langle -1,3\rangle\subset\mathbb{Q}_2^{*}/\mathbb{Q}_2^{*2}$. By Lemma~\ref{oddclassnmr} the order $k_3$ of
$[\mathfrak{p}_3]\in\cl_K$ is odd, so the product 
$x_3y_3$ yields $3\in \mathbb{Q}_2^{*}/\mathbb{Q}_2^{*2}$. Hence exactly one of $x_3, y_3$ after
$\mathfrak{p}$-adic completion has image $1$ or $-3$ in 
$\mathbb{Q}_2^{*}/\mathbb{Q}_2^{*2}$. As $\im_{\mathfrak{p}}(y_3)=\im_{\mathfrak{q}}(x_3)$, 
this implies that $x_3$ maps into $\langle -3\rangle \subset\mathbb{Q}_2^{*}/\mathbb{Q}_2^{*2}$ for precisely one of $\mathfrak{p},\mathfrak{q}$.
Denote this ideal by $\mathfrak{p}_2$, 
then $\mathfrak{p}_2$ is unramified in $K(\sqrt{x_3})$.\\
Let $x_2\in\mathfrak{p}_2$ be a generator for $\mathfrak{p}_2^{k_2}$, with $k_2$ the order of $[\mathfrak{p}_2]$. As above, multiplying $x_2$ by
$-1$ if necessary we may and will assume that $x_2$ maps $\mathfrak{q}_2$-adically into $\langle -3\rangle\subset\mathbb{Q}_2^{*}/\mathbb{Q}_2^{*2}$, where $\mathfrak{q}_2$ is the conjugate of $\mathfrak{p}_2$. Let $y_2$ be the conjugate of $x_2$, so $\mathfrak{q}_2^{k_2}=(y_2)$ and $x_2y_2=2^{k_2}$.\\

Proposition~\ref{sunitbasis} implies $K(S)=\langle -1,x_2,y_2,x_3,y_3\rangle$. We collect the local images in $K_{\mathfrak{p}}^{*}/K_{\mathfrak{p}}^{*2}$ of these generators, for $\mathfrak{p}\in S=\{\mathfrak{p}_2,\mathfrak{q}_2,\mathfrak{p}_3,\mathfrak{q}_3\}$, as follows.
\begin{equation}\label{redeitable23}
\begin{array}{c|c c c c}
 & \mathfrak{p}_2 & \mathfrak{q}_2 & \mathfrak{p}_3 & \mathfrak{q}_3 \\
\hline
-1 & -1 & -1 & -1 & -1 \\
x_2 & \cellcolor[gray]{0.65} & \cellcolor[gray]{0.65} & \cellcolor[gray]{0.8} & \cellcolor[gray]{0.8} \\
y_2 & \cellcolor[gray]{0.65} & \cellcolor[gray]{0.65} & \cellcolor[gray]{0.8} & \cellcolor[gray]{0.8} \\
x_3 & \cellcolor[gray]{0.8} & \cellcolor[gray]{0.8} & 3 & 1 \\
y_3 & \cellcolor[gray]{0.8} & \cellcolor[gray]{0.8} & 1 & 3 
\end{array}
\end{equation}
For $l\in\{2,3\}$ recall $\im_{\mathfrak{p}}(x_l)=\im_{\mathfrak{q}}(y_l)$ for conjugate $\mathfrak{p}$ and $\mathfrak{q}$ in $S$ and $x_ly_l=l^{k_l}$ with $k_l$ odd. Hence the  $2\times2$-block in the table corresponding to $x_l,y_l$ and conjugate $\mathfrak{p}, \mathfrak{q}$ is determined by any one entry in the block.\\

The normal closure of $K(\sqrt{x_2})/\mathbb{Q}$ yields a minimally ramified extension over $\mathbb{Q}(\sqrt{-2p})$, as $\mathfrak{p}_2$ is unramified in $K(\sqrt{y_2})$. 
Hence $y_2$ is $\mathfrak{p}_2$-adically a square if and only if $[-p,2,2]=1$. 
Thus the top left block in (\ref{redeitable23}) is determined by the R\'edei symbol $[2,2,-p]$. 
As suggested by the coloring, the two blocks away from the diagonal are both determined by the same R\'edei symbol. To see this, note that the normal closure of $K(\sqrt{x_3})/\mathbb{Q}$ yields a minimally ramified extension of $\mathbb{Q}(\sqrt{-3p})$. 
This extension has trivial inertia degree over $3$, hence $\im_{\mathfrak{p}_2}(x_3)=1$ if and only if $[-p,3,6]=1$. Similarly, the normal closure of
$K(\sqrt{x_2y_3})/\mathbb{Q}$ yields a minimally ramified extension of
$\mathbb{Q}(\sqrt{-6p})$.
Since $\im_{\mathfrak{p}_3}(y_3)=1$, 
this implies $\im_{\mathfrak{p}_3}(x_2)=1$ if and only if $[-p,6,3]=1$. 
Hence Table~(\ref{redeitable23}) is determined by the values of the two R\'edei symbols $[2,2,-p]$ and $[3,6,-p]$.
Below, the four possibilities for this pair
of symbols will be considered.

\begin{remark}
Since $(3,2)_3=-1$, the similar statement `$[-p,2,3]=[-p,3,2]$' cannot be used to show that the two lighter gray blocks in 
Table~(\ref{redeitable23}) are determined by the same R\'edei symbol. This is the reason for the workaround with the bottom right block. However, the proof of
$[a,b,c]=[a,b,c]$ relies on the product formula for quadratic Hilbert symbols in $\mathbb{Q}(\sqrt{a})$; there is nothing 
against using this product formula in $\mathbb{Q}(\sqrt{-p})$. 
Here the identity $\prod_{\mathfrak{p}}(x_2,x_3)_{\mathfrak{p}}=1$ leads to 
$\im_{\mathfrak{p}_2}(x_3)=1 \Leftrightarrow\im_{\mathfrak{p}_3}(x_2)=1$, but one still needs the symbol $[3,6,-p]$ to link the two blocks to splitting behaviour of primes in a \textit{fixed} (i.e., not depending on $p$) number field.
\end{remark}

For the Selmer group computations, observe that $S^2(J/K)=A\oplus \im(J(K)[2])$ for
\[
A=\{(e_1,\dotsc,e_5)\in S^2(J/K):e_3\xmapsto{\mathfrak{p}_3} 1\text{ and }e_4\xmapsto{\mathfrak{p}_2} 1\}.
\]
First consider the case $[2,2,-p]=[3,6,-p]=1$,
which means the table is as follows.
\begin{equation*}
\begin{array}{c|c c c c}
 & \mathfrak{p}_2 & \mathfrak{q}_2 & \mathfrak{p}_3 & \mathfrak{q}_3 \\
\hline
-1 & -1 & -1 & -1 & -1 \\
x_2 & \cellcolor[gray]{0.65}2 & \cellcolor[gray]{0.65}1 & \cellcolor[gray]{0.8}1 & \cellcolor[gray]{0.8}-1 \\
y_2 & \cellcolor[gray]{0.65}1 & \cellcolor[gray]{0.65}2 & \cellcolor[gray]{0.8}-1 & \cellcolor[gray]{0.8}1 \\
x_3 & \cellcolor[gray]{0.8}1 & \cellcolor[gray]{0.8}3 & 3 & 1 \\
y_3 & \cellcolor[gray]{0.8}3 & \cellcolor[gray]{0.8}1 & 1 & 3 
\end{array}
\end{equation*}
Let $x=(e_1,\dotsc,e_5)\in A$. The
$\mathfrak{p}_3$-adic and $\mathfrak{q}_3$-adic image implies
$e_3\in\langle x_2,-y_2\rangle$,
and therefore $\im_{\mathfrak{p}_2}(e_3)\subset\langle -1,2\rangle$ and $\im_{\mathfrak{q}_2}(e_3)\subset\langle -2\rangle$. 
This removes the fifth row of the $\mathbb{Q}_2$-table from consideration. As $\im_{\mathfrak{p}_2}(e_4)=1$, one concludes $\im_{\mathfrak{p}_2}(x)$ is in the span of
\begin{equation*}
\begin{array}{r c c c c c l}
(&6, & 1, & -1, & 1, & -6 &), \\
(&6, & 3, & -2, & 1, & -1 &).
\end{array}
\end{equation*}
Together with $\im_{\mathfrak{p}_3}(e_2)\subset\langle -1\rangle$ this gives $e_2\in\langle y_2,y_3\rangle$. 
Next, $\im_{\mathfrak{q}_2}(e_2)\subset\langle 2\rangle$ and $\im_{\mathfrak{q}_2}(e_3)\subset\langle -2\rangle$ implies  $\im_{\mathfrak{q}_2}(x)$ is in the span of
\begin{equation*}
\begin{array}{r c c c c c l}
(&2, & 1, & 1, & -1, & -2 &), \\
(&3, & 2, & 1, & -6, & -1 &),
\end{array}
\end{equation*}
so $e_3\in\langle x_2\rangle$. Since $n=(1,y_3,x_2,1,x_2y_3)\in A$,
a complement inside $A$ of $\langle n\rangle$ is obtained by setting $e_3=1$. For $x$
in this complement $\im_{\mathfrak{p}_2}(x)$ is trivial, hence $e_i\in\langle y_2,x_3\rangle$ for all $i$, implying $\im_{\mathfrak{q}_2}(x)$ is in the span of $(6,2,1,6,2)$.
Then $e_1,e_4\in\langle y_2x_3\rangle$ and $e_2,e_5\in\langle y_2\rangle$. A nontrivial $\im_{\mathfrak{q}_2}(x)$ can only occur
for $e_1,e_2,e_4,e_5$ all $\neq 1$, so this complement is at most one dimensional. Since $(y_2x_3,y_2,1,y_2x_3,y_2)\in A$ 
one concludes that $A$ is two-dimensional, and $\dim_{\mathbb{F}_2}S^2(J/K)=6$.\\

In the remaining three cases (i.e.,  $[2,2,-p]$ and $[3,6,-p]$ not both $1$) 
the computation is analogous; for details see \cite[\S~3.4.2-4]{evink}. The results are as follows. 
\begin{equation*}
\begin{array}{c|c|c|c}
\hspace{-4pt}{ }[2,2,-p]\hspace{-4pt}{ } & 
{ }\hspace{-4pt}[3,6,-p]\hspace{-4pt}{ } & 
{ }\hspace{-4pt}
\dim_{\mathbb{F}_2}S^2(J/K)\hspace{-2pt}{ } & \text{additional generators} \\
\hline
1 & 1 & 6 & (1,y_3,x_2,1,x_2y_3),(y_2x_3,y_2,1,y_2x_3,y_2) \\
1 & -1 & 6 & (-y_2y_3,y_2,1,-y_2y_3,y_2),(1,-y_3,y_2,1,-y_2y_3) \\
-1 & 1 & 4 &  \text{ none }\\
-1 & -1 & 4 &  \text{ none } \\
\end{array}
\end{equation*}
With this one proves 
Theorem~\ref{Thm23mod24}:
\begin{proof}[Proof of Theorem~\ref{Thm23mod24}.]
Let $p\equiv 23\bmod 48$ be prime. Then $p\equiv 7\bmod 16$ so Example~\ref{redeiexample1}
shows $[2,2,-p]=-1$. The table above implies $\dim_{\mathbb{F}_2}S^2(J/\mathbb{Q}(\sqrt{-p}))=4$
and as a consequence $\rank\,J(\mathbb{Q}(\sqrt{-p}))=0$. Hence $\rank\,J_p(\mathbb{Q})=0$ by
equation~\eqref{eq:rankrelationtwists}. Since
$p\equiv 17\bmod 24$, Proposition~\ref{prop2.1} 
yields $\dim_{\mathbb{F}_2}S^2(J_p/\mathbb{Q})=6$ hence
the exact sequence \eqref{eq:rankshasel} shows $\Sh(J_p/\mathbb{Q})[2]\cong(\mathbb{Z}/2\mathbb{Z})^2$.
\end{proof}

\begin{remark}
Part of what is proven above is that $\dim_{\mathbb{F}_2}S^2(J/\mathbb{Q}(\sqrt{-p}))$ for primes $p\equiv 23\bmod 24$ depends only on the values of $[2,2,-p]$ and $[3,6,-p]$. Hence instead of the
provided calculations for an undetermined $p\equiv 23\bmod 24$
one may take a fixed prime for each of the four possibilities for
the pair of R\'{e}dei symbols, and use e.g. Magma \cite{magma} to compute the Selmer group for this prime. The smallest primes covering all cases are given in the table below.
\[
\begin{array}{c|c|c}
p & [2,2,-p] & [3,6,-p] \\
\hline
    191 & 1 & 1 \\
    47 & 1 & -1 \\
    167 & -1 & 1 \\
    23 & -1 & -1
\end{array}
\]
We use Magma in this way to obtain proofs of Theorems~\ref{Thm17mod24} and \ref{Thm1mod24}.
\end{remark}

\begin{proposition}\label{prop5.4}
For $K=\mathbb{Q}(\sqrt{p})$ with $p\equiv 17\bmod 24$ prime, $\dim_{\mathbb{F}_2}S^2(J/K)$ is completely determined by the R\'edei symbols $[2,2,p]$ and $[2,-1,p]$.
\end{proposition}
\begin{proof}
Let $\sigma_1,\sigma_2:K\hookrightarrow\mathbb{R}$ be the two real embeddings of $K$. Take a fundamental unit $\varepsilon\in\mathcal{O}_K^*$ with $\sigma_1(\varepsilon)>0$. 
Lemma~\ref{oddclassnmr} implies $\varepsilon\overline{\varepsilon}=-1$, hence there is a unique prime ideal $\mathfrak{p}_2\subset \mathcal{O}_K$ over $2$ that is unramified in $K(\sqrt{\varepsilon})$. Let $\mathfrak{q}_2$ be the conjugate
of $\mathfrak{p}_2$ and write $\mathfrak{p}_2^k=(x_2)$ where $k$ is the order of $[\mathfrak{p}_2]$ in $Cl_K$. Multiplying $x_2$ by $\pm\varepsilon$ if necessary we can
and will assume that $x_2$ has positive norm and moreover $\mathfrak{q}_2$ is unramified in $K(\sqrt{x_2})$. Let $y_2$ be the conjugate of $x_2$, so $x_2y_2=2^k$.
Put $S=\{\mathfrak{p}_2,\mathfrak{q}_2,(3),\sigma_1,\sigma_2\}$, then $K(S)=\langle -1,\varepsilon,x_2,y_2,3\rangle$.
The table of images in $K_v^*/{K_v^*}^2$ of the generators of $K(S)$ is as follows
(as before, $r^2=2i\in\mathbb{Q}_3(i)$).
\begin{equation*}
\begin{array}{c|c c c c c}
 & \mathfrak{p}_2 & \mathfrak{q}_2 & (3) & \sigma_1 & \sigma_2 \\
\hline
-1 & -1 & -1 & 1 & -1 & -1 \\
\varepsilon & \cellcolor[gray]{0.65} & \cellcolor[gray]{0.65} & r & 1 & -1 \\
x_2 & \cellcolor[gray]{0.8} & \cellcolor[gray]{0.8} & r & \cellcolor[gray]{0.65} & \cellcolor[gray]{0.65} \\
y_2 & \cellcolor[gray]{0.8} & \cellcolor[gray]{0.8} & r & \cellcolor[gray]{0.65} & \cellcolor[gray]{0.65} \\
3 & 3 & 3 & 3 & 1 & 1 \\
\end{array}
\end{equation*}
The $3$-adic images of $\varepsilon,x_2,y_2$ 
follow by observing that the inertia degree
of $3\mathbb{Z}$ in the normal closures of $K(\sqrt{x_2})$ and $K(\sqrt{\varepsilon})$ over $\mathbb{Q}$ equals $4$. 
As $\mathfrak{p}_2$ is unramified in $K(\sqrt{\varepsilon})$ and in $K(\sqrt{y_2})$, the normal closures over $\mathbb{Q}$ yield minimally ramified extensions. Hence $\im_{\mathfrak{p}_2}(\varepsilon)=1
\Leftrightarrow [p,-1,2]=1$ and $\im_{\mathfrak{p}_2}(y_2)=1
\Leftrightarrow [p,2,2]=1$ and $\im_{\sigma_1}(x_2)=1\Leftrightarrow [p,2,-1]=1$. 
R\'edei reciprocity completes the proof.
\end{proof}

Aided by Magma for the rightmost column,
one computes the following table.
\[
\begin{array}{c|c|c|c}
p & [2,2,p] & [2,-1,p] & \dim_{\mathbb{F}_2}S^2(J/\mathbb{Q}(\sqrt{p})) \\
\hline
    113 & 1 & 1 & 6\\
    17 & 1 & -1 & 4\\
    41 & -1 & 1 & 4 \\
    89 & -1 & -1 & 6
\end{array}
\]
From the above, Theorem~\ref{Thm17mod24}
readily follows:
\begin{proof}[Proof of Theorem~\ref{Thm17mod24}]
Take $p\equiv 17\bmod 24$ prime and put
$K=\mathbb{Q}(\sqrt{p})$.
Proposition~\ref{prop5.4} and the table
above show
$\dim_{\mathbb{F}_2}S^2(J/K)=4
\Leftrightarrow [2,2,p][2,-1,p]=-1$.
Tri-linearity of the R\'edei symbol implies that the latter condition is equivalent to $[2,-2,p]=-1$, which by Example~\ref{Ex3.6}
means $p$ is not completely split in $\mathbb{Q}(\sqrt[4]{2})$.
As remarked earlier, 
$\dim_{\mathbb{F}_2}S^2(J/K)=4
\Rightarrow \rank\, J(K)=0
\Leftrightarrow \rank\,J_p(\mathbb{Q})=0$.
Proposition~\ref{prop2.1}
and the exact sequence \eqref{eq:rankshasel}
now finish the proof.
\end{proof}

Lastly we cover the case $p\equiv 1\bmod 24$.
\begin{proposition}\label{prop5.5}
For $K=\mathbb{Q}(\sqrt{p})$ with $p\equiv 1\bmod 24$ prime, $\dim_{\mathbb{F}_2}S^2(J/K)$ is completely determined by the R\'edei symbols $[2,2,p],\;[2,-1,p],\;[3,-2,p]\text{ and }[3,6,p]$.
\end{proposition}
\begin{proof}
Let $p\equiv 1\bmod 24$ be prime and put
$K=\mathbb{Q}(\sqrt{p})$. As in the proof of
Proposition~\ref{prop5.4} let $\sigma_1,\sigma_2:K\hookrightarrow\mathbb{R}$ be the real embeddings, take a fundamental unit 
$\varepsilon\in\mathcal{O}_K$ with $\sigma_1(\varepsilon)>0$, 
let $\mathfrak{p}_2$ be the prime over $2$ that is unramified in $K(\sqrt{\varepsilon})$, 
and denote the conjugate of
$\mathfrak{p}_2$ by $\mathfrak{q}_2$. Then $\mathfrak{p}_2^{k_2}=(x_2)$ with $k_2=\mathrm{ord}([\mathfrak{p}_2])$, where one chooses 
$x_2\in\mathcal{O}_K$ of positive norm and such that $\mathfrak{q}_2$ is unramified in $K(\sqrt{x_2})$.\\
Let $\mathfrak{p}_3$ be the prime over $3$ that splits in $K(\sqrt{x_2})$, and let $\mathfrak{q}_3$ be its conjugate. 
With $k_3=\mathrm{ord}([\mathfrak{p}_3])$, write $\mathfrak{p}_3^{k_3}=(x_3)$ with $x_3\in\mathcal{O}_K$ of positive norm, chosen so that $\mathfrak{p}_2$ is unramified in $K(\sqrt{x_3})$. For $i\in\{2,3\}$ let $y_i$ be the conjugate of $x_i$, so $x_iy_i=i^{k_i}$. 
Put $S=\{\mathfrak{p}_2,\mathfrak{q}_2,\mathfrak{p}_3,\mathfrak{q}_3,\sigma_1,\sigma_2\}$, then $K(S)=\langle -1,\varepsilon,x_2,y_2,x_3,y_3\rangle\subset K^*/{K^*}^2$. 
Information on local images  of $K(S)$ is presented in the following table.
\begin{equation*}
\begin{array}{c|c c c c c c}
 & \mathfrak{p}_2 & \mathfrak{q}_2 & \mathfrak{p}_3 & \mathfrak{q}_3 & \sigma_1 & \sigma_2 \\
\hline
-1 & -1 & -1 & -1 & -1 & -1 & -1 \\
\varepsilon & \cellcolor[gray]{0.4} & \cellcolor[gray]{0.4} & \cellcolor[gray]{0.55} & \cellcolor[gray]{0.55} & 1 & -1 \\
x_2 & \cellcolor[gray]{0.7} & \cellcolor[gray]{0.7} & 1 & -1 & \cellcolor[gray]{0.4} & \cellcolor[gray]{0.4} \\
y_2 & \cellcolor[gray]{0.7} & \cellcolor[gray]{0.7} & -1 & 1 & \cellcolor[gray]{0.4} & \cellcolor[gray]{0.4} \\
x_3 & 1 & 3 & \cellcolor[gray]{0.85} & \cellcolor[gray]{0.85} & \cellcolor[gray]{0.55} & \cellcolor[gray]{0.55} \\ 
y_3 & 3 & 1 & \cellcolor[gray]{0.85} & \cellcolor[gray]{0.85} & \cellcolor[gray]{0.55} & \cellcolor[gray]{0.55} 
\end{array}
\end{equation*}
To see this, first consider the bottom middle $2\times 2$ block. Note that $[p,6,3]=1$ if and only if $\im_{\mathfrak{p}_3}(y_3)=\im_{\mathfrak{p}_3}(x_2y_3)=1$,
and similarly $[p,3,6]=1$ precisely when the equivalence $\im_{\mathfrak{q}_2}(y_3)=1$ $\Leftrightarrow$ $\im_{\mathfrak{p}_3}(y_3)=1$ holds. Since $[p,6,3]=[p,3,6]$, this implies $\im_{\mathfrak{q}_2}(y_3)=1$ and moreover $\im_{\mathfrak{p}_3}(y_3)=1$ if and only if $[p,6,3]=1$.
The choice of $x_3$ and the equality $x_3y_3=3^{k_3}$ for $k_3=\ord([\mathfrak{p}_3])$ odd, implies the bottom left block.
The remaining assertions about the table (in particular:
the regions colored in the same shade of grey are determined
by any one entry in that region) are straightforward and/or
analogous to what we did in other 
$\bmod~24$ cases.\\
As in the $17\bmod 24$ case,  $\im_{\mathfrak{p}_2}(\varepsilon)=1$ $\Leftrightarrow$ $[p,-1,2]=1$, and $\im_{\sigma_1}(x_2)=1$ $\Leftrightarrow$ $[p,2,-1]=1$, and $\im_{\mathfrak{p}_2}(y_2)=1$ $\Leftrightarrow$ $[p,2,2]=1$.\\
Finally, $\im_{\mathfrak{p}_3}(\varepsilon)=\im_{\mathfrak{p}_3}(\varepsilon x_2)=1$ precisely when $[p,-2,3]=1$. 
Since $\im_{\mathfrak{p}_2}(x_3)=1$, one has $\im_{\sigma_1}(x_3)=1\Leftrightarrow [p,3,-2]=1$. 
R\'edei reciprocity finishes the proof.
\end{proof}
Using Magma for the rightmost column results in the following table (in fact implying a
stronger version of Proposition~\ref{prop5.5}:
$\dim_{\mathbb{F}_2}S^2(J/\mathbb{Q}(\sqrt{p}))$ for the primes
$p\equiv 1\bmod 24$ only depends on the R\'edei symbols
$[2,2,p],\,[2,-1,p]$, and $[3,-2,p]$).
\[
\begin{array}{c|c|c|c|c|c}
p & [2,2,p] & [2,-1,p] & [3,-2,p] & [3,6,p] & \dim_{\mathbb{F}_2}S^2(J/\mathbb{Q}(\sqrt{p})) \\
\hline
    2593 & 1 & 1 & 1 & 1 & 8\\
    1153 & 1 & 1 & 1 & -1 & 8\\
    337 & 1 & 1 & -1 & 1 & 4\\
    557 & 1 & 1 & -1 & -1 & 4\\
    433 & 1 & -1 & 1 & 1 & 4\\
    97 & 1 & -1 & 1 & -1 & 4\\
    241 & 1 & -1 & -1 & 1 & 6\\
    193 & 1 & -1 & -1 & -1 & 6\\
    1321 & -1 & 1 & 1 & 1 & 6\\
    409 & -1 & 1 & 1 & -1 & 6\\
    1129 & -1 & 1 & -1 & 1 & 4\\
    313 & -1 & 1 & -1 & -1 & 4\\
    937 & -1 & -1 & 1 & 1 & 6\\
    1033 & -1 & -1 & 1 & -1 & 6\\
    73 & -1 & -1 & -1 & 1 & 4\\
    601 & -1 & -1 & -1 & -1 & 4\\
\end{array}
\]

\begin{proof}[Proof of Theorem~\ref{Thm1mod24}]
Let $p\equiv 1\bmod 24$ be prime and put
$K=\mathbb{Q}(\sqrt{p})$. Proposition~\ref{prop2.1} implies $\dim_{\mathbb{F}_2}S^2(J_p/\mathbb{Q})=8$, 
hence as in the proofs of
Theorems~\ref{Thm23mod24} and \ref{Thm17mod24} 
it suffices to show that $\dim_{\mathbb{F}_2}S^2(J/K)=4$
if $p$ satisfies one of the conditions (a), (b), or (c)
mentioned in the statement of Theorem~\ref{Thm1mod24}.
\\
Note: $p$ splits completely in $\mathbb{Q}(\sqrt[4]{2}) \Leftrightarrow [2,2,p][2,-1,p]=[2,-2,p]=1$.
Also,  $p$ splits completely in $\mathbb{Q}(\sqrt{1+\sqrt{3}})\Leftrightarrow [3,-2,p]=1$,
and $[2,2,p]=1\Leftrightarrow p\equiv 1\bmod 16$. Hence condition (a) corresponds to the cases
$p\in\{73, 337, 557, 601\}$ in the table above.
Condition (b) corresponds to $p\in\{97, 433\}$ in the table,
and condition~(c) to $p\in\{313, 1129\}$.
In all these cases the table shows $\dim_{\mathbb{F}_2}S^2(J/K)=4$, hence the result follows
by using Proposition~\ref{prop5.5}.
\end{proof}

We finish this section by presenting an analogous result for elliptic curves; we restrict to $p\equiv 1\bmod 24$ but 
in the same spirit one obtains similar statements for the other congruence classes
$p\bmod 24$.

\begin{proposition}\label{ellRedeiexample}
Let $E/\mathbb{Q}$ be an elliptic curve with good reduction 
away from $2, 3$ and with $E(\mathbb{Q})[2]=E(\overline{\mathbb{Q}})[2]$. For a prime $p\equiv 1\bmod 24$, the size of the $2$-Selmer group $S^2(E/\mathbb{Q}(\sqrt{p}))$ is determined by $E/\mathbb{Q}$ together with the Rédei symbols
\[
[2,2,p],\;[2,-1,p],\;[3,-2,p],\;[3,6,p].
\]
\end{proposition} 
\begin{proof}
We use the notation introduced in the proof of
Proposition~\ref{prop5.5}.
Descent yields an embedding
\[
\delta\colon E(K)/2E(K)\hookrightarrow{}\ker\left(\bigoplus_{i=1}^3K(S)\to K(S)\right)
\]
and $S^2(E/K)$ consists of the elements in
$\bigoplus_{i=1}^3K(S)$ that locally are in the image of the
corresponding maps $\delta_v$, for all
$v\in S=\{\sigma_1,\sigma_2,\mathfrak{p}_2,\mathfrak{q}_2,
\mathfrak{p}_3,\mathfrak{q}_3\}$.
For these $v$, the image in $K_v^*/{K_v^*}^2$ of a basis
for $K(S)$ is described in the table presented in the proof of
Proposition~\ref{prop5.5}. As this table is determined
by the four given
R\'edei symbols and $S^2(E/K)$ consists of
triples of elements in $K(S)$ that 
for $v\in S$ locally are in $\delta_v(E(K_v))$, the result follows.
\end{proof}
\begin{remark}
The finite list of elliptic curves satisfying the conditions from
Proposition~\ref{ellRedeiexample} was already presented in the
PhD thesis of F.B.~Coghlan \cite{Coghlan1966}. In fact he listed
{\em all} elliptic curves over $\mathbb{Q}$ having good reduction away
from $2$ and $3$. Precisely $28$ of these have full rational
$2$-torsion. In the LMFDB tables \cite{lmfdb} contain them under
the conductors
$
\left\{ 24,\; 32,\; 48,\; 64,\; 72,\; 96,\; 144,\; 192,\; 288,\; 576  \right\}
$.
\end{remark} 
\section{The $\mathbb{Q}$-rational points}\label{Sect6}

Here briefly $\mathbb{Q}$-rational points of the curves
$C_p$ are discussed. The proof of Corollary~\ref{cor2.2}
shows that for primes $p$ such that
$\rank\,J_p(\mathbb{Q})=0$, the set $C_p(\mathbb{Q})$
consists of the Weierstrass points only. Below a less
immediate case is discussed, namely a situation
with $\rank\,J_p(\mathbb{Q})=2$. We remark that in this case $\rank\,J_p(\mathbb{Q})$ is not strictly smaller than the genus of $C_p$ so the standard Chabauty method does not apply.

\vspace{\baselineskip}

Take the prime $p=241$. 
Using $\rank\,J_p(\mathbb{Q})\leq
\dim_{\mathbb{F}_2}S^2(J/\mathbb{Q}(\sqrt{p})))-4$,
the row $p=241$ in the table preceding the proof of Theorem~\ref{Thm1mod24} yields $\rank\,J_{241}(\mathbb{Q})\leq 2$. 
The Mumford representations
\begin{align*}
P=&\left(x^2 - \tfrac{868230159329}{1782528400}x + \tfrac{8609056225}{4456321},
             \tfrac{83127269153329233}{75258349048000}x
              - \tfrac{8905877454269565}{37629174524}\right),\\
     Q=&\left(x^2 - \tfrac{692452}{3721}x + \tfrac{73966756}{3721},
             \tfrac{6990522627}{2269810}x + \tfrac{1284886465269}{1134905}\right)
\end{align*}
turn out to
define points in $J_{241}(\mathbb{Q})$. The homomorphism
$\delta\colon J_p(\mathbb{Q})\to S^2(J_p/\mathbb{Q})$ yields
$\delta(P)=(2,p,1,p,2)$ and $\delta(Q)=(1,p,p,p,p)$. These images
are independent of $\delta(J_p(\mathbb{Q})_{\mbox{\scriptsize tor}})$
which is generated by $(6,-p,-2p,-3p,-p)$, $(p,-6,-p,-2p,-3p)$, $(2p,p,1,-p,-2p)$ and $(3p,2p,p,-6,-p)$.
Hence $\rank\,J_{241}(\mathbb{Q})=2$. Moreover by
Proposition~\ref{prop2.1} and equality \eqref{eq:rankshasel}
one concludes $\Sh(J_{241}/\mathbb{Q})[2]\cong(\mathbb{Z}/2\mathbb{Z})^2$.

To determine $C_{241}(\mathbb{Q})$ the methods developed in
\cite{BrSt} will now be used. Although this works in much greater
generality, here it is only briefly recalled in the special case
of the curves $C_p$. Consider the composition
\[
C_p(\mathbb{Q})\longrightarrow J_p(\mathbb{Q})
\stackrel{\delta}{\longrightarrow} S^2(J_p/\mathbb{Q})
\]
mapping $(a,b)\in C_p(\mathbb{Q})$ with $b\neq 0$ to
$(a+2p,a+p,a,a-p,a-2p)\in S^2(J_p/\mathbb{Q})$.
Is $s=(e_1,\ldots,e_5)\in S^2(J_p/\mathbb{Q})$, then being in the
image of $C_p(\mathbb{Q})$ implies that one has a rational point
on the smooth, complete curve $X_s/\mathbb{Q}$ corresponding to the 
affine equations
\[
x+2p=e_1y_1^2,\;x+p=e_2y_2^2,\;
x=e_3y_3^2,\;x-p=e_4y_4^2,\;x-2p=e_5y_5^2.
\]
Here by abuse of notation $e_j$ represents the class $e_j\in\mathbb{Q}^*/{\mathbb{Q}^*}^2$; the result is independent of this representing element. The curve $X_s$ is what in \cite{BrSt} is called a two-cover of $C_p$ over $\mathbb{Q}$.
The ``Two-Selmer set'' of $C_p/\mathbb{Q}$ is
\[
\left\{ s\in S^2(J_p/\mathbb{Q})\;:\;
X_s\;\text{has rational points everywhere locally}
\right\}.
\]
As an example, for $p=241$ let $s:=\delta(P)=(2,p,1,p,2)$. Among the equations
for $X_s$ one has $x+2p=2y_1^2$ and $x-p=py_4^2$, defining the conic 
$Q\colon 2y_1^2-py_4^2=3p$. One obtains a finite
morphism $X_s\to Q$ defined over $\mathbb{Q}$.
Since $Q(\mathbb{Q}_2)$ (as well as $Q(\mathbb{Q}_3)$) is empty, this shows $\delta(P)$
is not in the Two-Selmer set of $C_p/\mathbb{Q}$.
In other words: although $\delta(P)$ is everywhere
locally (even globally!) in $\delta_v(J_p(\mathbb{Q}_v))$, it is not in the
image of $C_p(\mathbb{Q}_2)\subset J_p(\mathbb{Q}_2)$.

The Magma command {\tt TwoCoverDescent();} computes
the curves $X_s$ corresponding to the Two-Selmer set. In our case it turns out that of the $2^8$
elements in $S^2(J_{241}/\mathbb{Q})$, only the
six $\delta([W]-[\infty])$ for $W\in C_p(\mathbb{Q})$ a Weierstrass point, are in the
Two-Selmer set. We now show that for each of these
six elements $s$ one finds that
$\left\{R\in C_p(\mathbb{Q})\;:\;\delta([R]-[\infty])
=s\right\}$ consists of only a Weierstrass point.
As a consequence, $C_{241}(\mathbb{Q})=\{\infty,\,
(0,0),\,(\pm 241,0), (\pm 482,0)\}$. We use the notation
$D_\xi$ (here for certain elements in $J_p$) as introduced
on page~\pageref{NotationDxi}.\\
\begin{itemize}
    \item $s:=\delta(0)=(1,1,1,1,1)$.
    If $(a,b)\in C_{241}(\mathbb{Q})$ with $b\neq 0$
    would result in $\delta$-image $s$,
    then in particular the elliptic curve 
    $E_1\colon y^2=x(x+p)(x+2p)$ admits a point
    in $E_1(\mathbb{Q})$ with $x=a$ and $y\neq 0$.
    Since $E_1(\mathbb{Q})\cong \mathbb{Z}/2\mathbb{Z}\times \mathbb{Z}/2\mathbb{Z}$, no such point exists.
    \item $s:=\delta(D_{-2p})=(6,-p,-2p,-3p,-p)$.
    In this case, considering the 1st, 3rd, and 4th
    entry results in the elliptic curve
    $E_2\colon y^2=x(x+2p)(x-p)$ satisfying
    $E_2(\mathbb{Q})\cong \mathbb{Z}/2\mathbb{Z}\times \mathbb{Z}/2\mathbb{Z}$. Hence only
    the Weierstrass point $(-482,0)\in C_p(\mathbb{Q})$ yields $\delta$-image $s$.
    \item $s:=\delta(D_{-p})=(p,-6,-p,-2p,-3p)$.
    Here the 2nd, 4th, and 5th entry results in
    the elliptic curve $E_3\colon -y^2=(x+p)(x-p)(x-2p)$ whose only rational points
    are the points of order at most $2$. Reasoning
    as before, this implies that only
    the Weierstrass point $(-241,0)\in C_p(\mathbb{Q})$ yields $\delta$-image $s$.
    
    \item $s:=\delta(D_0)=(2p,p,1,-p,-2p)$.
    Using entries 1, 2, and 3 results in the
    elliptic curve $E_4\colon 2y^2=x(x+p)(x+2p)$,
    whose only rational points
    are the points of order at most $2$. 
    As above, this implies that only
    the Weierstrass point $(0,0)\in C_p(\mathbb{Q})$ yields $\delta$-image $s$.
    \item $s:=\delta(D_{p})=(3p,2p,p,-6,-2p)$.
    Here we use entries 1, 2, and 4, leading
    to $E_5\colon -y^2=(x+2p)(x+p)(x-p)$.
    Also here the only rational points are the points of order dividing $2$. So 
    only
    the Weierstrass point $(241,0)\in C_p(\mathbb{Q})$ yields $\delta$-image $s$.
    \item $s:=\delta(D_{2p})=(p,3p,2p,p,6)$.
    Using entries 1, 2, and 5 one obtains
    $E_6\colon 2y^2=(x+2p)(x+p)(x-2p)$. Here as well, the only rational points 
    are the points of order dividing $2$. So 
     $(482,0)\in C_p(\mathbb{Q})$ is the
    only rational point with $\delta$-image $s$.
\end{itemize}
This completes the determination of the rational
points on $C_{241}$.\\
Note that for $p=5$ there are two additional points: one has $\#C_5(\mathbb{Q})=8$, where the two non-Weierstrass points are $(20,\pm 1500)$. Applying Chabauty's method implies that there are no other points.\\
It may be possible to extend the method described here and in this
way answer the question whether a prime $p>5$ exists
such that $\#C_p(\mathbb{Q})>6$. 

As a final remark, recall that the two-cover $X:=X_{(1,1,1,1,1)}$ of $C_p/\mathbb{Q}$
corresponds to the affine model
\[
x+2p=y_1^2,\;x+p=y_2^2,\;
x=y_3^2,\;x-p=y_4^2,\;x-2p=y_5^2.
\]
The maps $y_j\mapsto -y_j$ define a group
$(\mathbb{Z}/2\mathbb{Z})^5$ in 
$\text{Aut}_{\mathbb{Q}}(X)$. Using
appropriate subgroups one obtains up to
isogeny the decomposition of $\text{Jac}(X)$ over $\mathbb{Q}$ given as follows.
Let $E_{24}\colon y^2=(x-1)(x^2-4)$ and $E_{32}\colon y^2=x^3-x$ and finally 
$E_{96a}\colon y^2=x(x+1)(x-2)$ be elliptic
curves over $\mathbb{Q}$. For any such
$E/\mathbb{Q}$ and any $d\in\mathbb{Q}/{\mathbb{Q}^*}^2$
we write $E^{(d)}$ for the quadratic twist
of $E$ defined by $d$.
Then $\text{Jac}(X)$ is isogenous over
$\mathbb{Q}$ to the product
\[
J_p \times \left(E_{24}\right)^2 \times
E_{24}^{(-1)} \times E_{24}^{(p)} \times 
E_{24}^{(-p)} \times
\left(E_{32}^{(p)}\right)^3 \times E_{32}^{(2p)}
\times
\left(E_{96a}^{(-2)}\right)^2 \times
\left(E_{96a}^{(p)}\right)^2 \times
\left(E_{96a}^{(-p)}\right)^2.
\]
In particular the rank of $\text{Jac}(X)$ is determined
by that of $J_p$ and of the given twists of the 
three elliptic curves
$E_{24}$, $E_{32}$, and $E_{96a}$.
Using analogs of Proposition~\ref{ellRedeiexample}
for various classes of primes $p$ provides
a natural approach towards bounding
$\rank\,\text{Jac}(X)(\mathbb{Q})$.

\section*{Acknowledgement}
Several people provided valuable suggestions during this work. We mention 
Peter Stevenhagen who in the context of \cite{StrTo1994}
already in the 90's showed one of us
 relations
with $4$-ranks of class groups, and much more recently referred us to his work on R\'edei symbols.
We also mention Nils Bruin who explained
one of us the Magma implementation of the
elliptic Chabauty method, originally while
working on \cite[Section~3]{ChaTop} but 
 also very relevant for the much simpler situation given in Section~\ref{Sect6} 
 of the present paper.
Michael~Stoll directed us to the theory of Two-Selmer sets and envisioned the result
given in Section~\ref{Sect6}.
Finally, Steffen M\"{u}ller,  Jeroen~Sijsling, and Marius~van~der~Put all
showed their interest in this work and in that way encouraged us to complete it.
\printbibliography

\end{document}